\documentclass[12pt,reqno]{amsart}
\usepackage[utf8]{inputenc}

\usepackage{fullpage}

\usepackage{mathrsfs}
\usepackage{hyperref}
\usepackage{cite}
\usepackage{mathtools,amsmath,amssymb}
\usepackage{enumitem}
\usepackage{physics}
\usepackage{comment}

\newtheorem{lem}{Lemma}
\newtheorem{thm}{Theorem}
\newtheorem{prop}{Proposition}
\newtheorem{rmk}{Remark}


\title{Higher Mertens constants for almost primes II}
\author[J.\ Bayless]{Jonathan Bayless}

\address{University of Maine at Augusta, 46 University Drive, Augusta, Maine 04330}

\email{jonathan.bayless@maine.edu}

\author[P.\ Kinlaw]{Paul Kinlaw}

\address{Dickinson College, 28 North College Street, PO Box 1773, Carlisle, PA 17013-2896}

\email{paulkinlaw1@gmail.com}

\author[J.\ D.\ Lichtman]{Jared Duker Lichtman}

\address{Mathematical Institute, University of Oxford, Oxford, OX2 6GG, UK}

\email{jared.d.lichtman@gmail.com}

\date{\today}

\begin{document}

\maketitle

\begin{abstract}
For $k\ge1$, let $\mathcal{R}_k(x)$ denote the reciprocal sum up to $x$ of numbers with $k$ prime factors, counted with multiplicity. In prior work, the authors obtained estimates for $\mathcal{R}_k(x)$, extending Mertens' second theorem, as well as a finer-scale estimate for $\mathcal{R}_2(x)$ up to $(\log x)^{-N}$ error for any $N > 0$. In this article, we establish the limiting behavior of the higher Mertens constants from the $\mathcal{R}_2(x)$ estimate. We also extend these results to $\mathcal{R}_3(x)$, and we remark on the general case $k\ge4$.
\end{abstract}

\section{Introduction}

Let $\Omega(n)$ denote the number of prime factors of an integer $n$, counted with multiplicity. In \cite[Theorem 1.1]{HME}, the authors established the asymptotic estimates, for any $k\ge1$,
\begin{equation}\label{Vk}
{\mathcal R}_k(x) := \mathop{\sum_{n\le x}}_{\Omega(n)=k}\frac{1}{n} \ = \ \sum_{j=0}^k\frac{\nu_{k-j}}{j!}(\log_2 x)^j \ + \ O_k\left(\frac{(\log_2 x)^{k-1}}{\log x}\right),
\end{equation}
for explicit constants $\nu_j$. Here $\log_2 x=\log\log x$.
%where $\nu_j$ are the Taylor coefficients of $\nu(z) = \frac{1}{\Gamma(z+1)}\prod_{p}(1-\frac{z}{p})^{-1}(1-\frac{1}{p})^z$. Here $\log_2 x=\log\log x$ denotes the iterated natural logarithm. 
%The proof of \eqref{Vk} in \cite{HME} differs from the classic Sathe-Selberg theorem \cite[Theorem II.6.5]{TenenbaumBook}, which does not give the constant term in \eqref{Vk}.
The case $k=1$ follows classically from  Mertens' second theorem,
\begin{equation}\label{eq:mertens}
{\mathcal R}_1(x):=\sum_{p\le x}\frac{1}{p} \ = \ \log_2 x + \beta  + O\left(\frac{1}{\log x}\right),
\end{equation}
where $\beta = \gamma  + \sum_p(1/p + \log(1-1/p))=0.2614\cdots$ is the (Meissel-)Mertens constant. %From \eqref{Vk}, $\nu_k$ may be viewed as higher analogues of the Mertens constant $\nu_1=\beta$.

When $k=1$, a strong form of the prime number theorem implies \eqref{eq:mertens}  with finer-scale error $O_N\left((\log x)^{-N}\right)$ for any $N > 0$. When $k=2$, \cite[Theorem 1.3]{HME} gives a finer-scale estimate of the same quality.  Specifically, for any $N> 0$, we have
\begin{align}\label{eq:R2finer}
{\mathcal R}_2(x) \, = \ \frac{1}{2}(\log_2 x+\beta)^2+\frac{P(2)-\zeta(2)}{2}+\sum_{1\le j< N}\frac{\alpha_j}{\log^j x}+O_N\left((\log x)^{-N}\right).
\end{align}
Here $\zeta(s)=\sum_n n^{-s}$ and $P(s)=\sum_p p^{-s}$ denote the zeta and prime zeta functions, respectively, and the coefficients $\alpha_j$ are given by
\begin{equation}\label{alphalimit}
\begin{aligned}
\alpha_j & = \lim_{x\to\infty}\frac{1}{j}\left(\frac{\log^j x}{j} - \sum_{p\le x}\frac{\log^j p}{p}\right).
\end{aligned}
\end{equation}

In \cite{HME}, it was conjectured that the constants satisfy $\alpha_j\sim j!\,2^{j}/(2j^2)$. Also see \cite[Lemma 8.3]{HME}. In this article, we prove this asymptotic with quantitative error.

\begin{thm} \label{mainthm}
For any $m\ge 0$, we have
\begin{equation} \label{asym_alpha_j}
\alpha_j = \frac{j!\,2^{j}}{2j^2} \left( 1 + O_m\left( j^{-m}\right) \right). 
\end{equation}
Moreover, assuming the Riemann hypothesis, we may replace the error $O_m\left( j^{-m}\right)$ above by
$(2/3+o(1))\left(3/4\right)^j$ (see \eqref{pilibound}).
\end{thm}

See Section \ref{compute} for computations illustrating the rate of convergence.  

We will prove Theorem \ref{mainthm} in Section \ref{mainthmsec}.  The proof makes use of the following constants, for $j\ge 0$ and $a > 1$,
\begin{equation}\label{alphajk}
\alpha_{j,a} := (-1)^j P^{(j)}(a) = \sum_p \frac{(\log p)^j}{p^a},
\end{equation}
recalling $P(s) = \sum_p p^{-s}$. Also, for each integer $j\ge 1$, we define $\alpha_{j,1} : = -j\alpha_j$.  We note that the quantity $\alpha_{j,1}$ also appears in \cite{cricsan2021counting} where it is denoted as $B_j$.

Moreover, we extend \eqref{eq:R2finer} to obtain finer-scale asymptotics for $\mathcal{R}_3(x)$.

\begin{thm}\label{R3thm1}
For any $N > 0$, we have
\begin{displaymath}
\begin{aligned}
   \mathcal{R}_3(x) &=  \frac{1}{6}(\log_2 x+\beta)^3+\frac{P(2)-\zeta(2)}{2}(\log_2 x+\beta)+\frac{P(3)+\zeta(3)}{3}\\
   &\quad +\sum_{1\le j< N}\frac{\alpha_j(\log_2 x+\beta)-r_j}{\log^j x} \ + \ O_N\left(\frac{\log_2 x}{(\log x)^N}\right),
\end{aligned}
\end{displaymath}
where
\begin{displaymath}
r_j = \frac{2^{j-1}\alpha_{j,2}}{j} + \sum_{i=1}^{j-1} \left(\frac{\alpha_j}{i}+i\binom{j-1}{i}\frac{\alpha_i\alpha_{j-i}}{2}\right).
\end{displaymath}
\end{thm}
In particular, $r_1 = \alpha_{1,2}, r_2 = \alpha_{2,2}+\alpha_2+\alpha_1^2/2$, and $r_3 = 4\alpha_{3,2}/3+3\alpha_3/2+2\alpha_1\alpha_2$.  Moreover, we show that $r_j$ satisfies
\begin{equation}\label{rnestimate}
    r_j = \frac{3\alpha_j}{2}(\log j + O(1)).
\end{equation}

The case $N=1$ of Theorem \ref{R3thm1} already gives a slight refinement of \cite[Theorem 2.3]{HME} (the error terms are improved by a factor of $\log_2 x$) and so the sums over $j < N$ are of most interest. We will prove Theorem \ref{R3thm1} in Section \ref{R3sec}.

The proof of Theorem \ref{R3thm1} employs a result of Tenenbaum \cite[Theorem 1.1]{tenenbaum2016generalized} for a related sum. He showed that for any $N>0$,
\begin{displaymath}
\mathcal{S}_k(x) := \sum_{p_1\cdots p_k\le x}\frac{1}{p_1\cdots p_k} = \sum_{0\le j< N}\frac{S_{k,j}(\log_2 x)}{\log^j x} + O_N\left(\frac{(\log_2 x)^{k-1}}{(\log x)^N}\right)
\end{displaymath}
for certain explicit polynomials $S_{k,j}$. This generalized work \cite{popa2014double,popa2016triple} of Popa, who addressed the cases $k=2$ and $3$ and $N=1$.  We note that Qi--Hu \cite{qi2019multiple} gave an elementary proof of the case $N=1$ of Tenenbaum's result for general $k$.

In principle, one may determine fine-scale asymptotics for $\mathcal{R}_k(x)$ for general $k\ge4$, with greater patience. We give a proof-of-concept for $\mathcal{R}_4(x)$ in Theorem \ref{R4thm} below.

\begin{rmk}
Note the constants $\alpha_{j,2}$ appear in Theorem \ref{R3thm1}. More generally, in the finer-scale estimates for $\mathcal{R}_k(x)$ we find a coefficient of $-(k-1)^{j-1}\alpha_{j,k-1}/j$ contributing to the $1/\log^j x$ term. See Proposition \ref{pqa} below.
\end{rmk}

The method of proof begins from the the fine-scale asymptotic for $\mathcal{S}_k(x)$, strengthening \cite[Proposition 3.3]{HME} as follows. One must evaluate partial sums over primes $p$ involving expressions of the type
\begin{displaymath}
\frac{(\log_2 (x/p^a))^b}{p^a\log^c (x/p^a)}
\end{displaymath}
for integers $b\ge 0$ and $c\ge 0$, with error $\ll_N (\log x)^{-N}$. For this we use the expansions,
\begin{equation}\label{logexpand}
\frac{1}{\log\frac{x}{t^a}} = \sum_{m\ge 1}\frac{(a\log t)^{m-1}}{(\log x)^m} = \sum_{1\le m < N}\frac{(a\log t)^{m-1}}{(\log x)^m}+O_{a,N}\left(\frac{(\log t)^{N} t}{(\log x)^{N}}\right),
\end{equation}
for any integers $a\ge 2$ and $N > 0$, and
\begin{equation}\label{logexpand2}
\log_2\frac{x}{t^a} = \log_2 x - \sum_{m\ge 1}\frac{(a\log t)^m}{m(\log x)^m} = \log_2 x - \sum_{1\le m < N}\frac{(a\log t)^m}{m(\log x)^m}+O_{a,N}\left(\frac{\log^{N} t}{\log^{N}x}\right)
\end{equation}
for all sufficiently large $x$, and these estimates are uniform for all $t$ in the necessary range, say, $t\in [2,(x/2)^{1/a}]$. Then we use these expansions with $t=p$, apply the multinomial theorem to expand the $b^\textnormal{th}$ and $c^\textnormal{th}$ powers, respectively. The partial sums with $t=p$ lead to the coefficients $\alpha_{j,a}$ in \eqref{alphajk}.

In the course of the proof of Theorem \ref{mainthm}, we also obtain some explicit bounds for polylogarithms.  These may be of independent interest.

\iffalse
In light of the work for $k\le 4$, we make the following conjecture.  For each $k\ge 1$ and each $N\ge 0$, we have

\begin{equation}
\mathcal{R}_k(x)= R_k(\log_2 x+\beta)+\sum_{1\le j<N}\sum_{2\le m\le k}\frac{c_{m,j}R_{k-m}(\log_2 x+\beta)}{\log^j x}+O_{k,N}\left(\frac{(\log_2 x)^{k-1}}{\log^{N}x}\right),
\end{equation}
where the $c_{m,j}$ are constants that can be expressed in terms of the $\alpha$.  In particular, $c_{2,j}=\alpha_j$ for all $j\ge 1$.  

\begin{rmk}
Note that when $k=1$, the double sum is empty and we recover Mertens' second theorem with error term $\ll_A (\log x)^{-A}$ for any $A$.  When $N=0$ we recover \cite[Thm.\ 1.1]{HME}, and when $k=2$ we recover \cite[Thm.\ 1.3]{HME}.  
\end{rmk}

By convention, we define $R_0(X)=1$ for all $X\ge 1$, which is justified by the trivial observation that $1$ is the only product of zero primes.
\fi

\section{Preliminary Lemmas}

We will use several results, including the prime number theorem in the form
\begin{equation} \label{pnt}
\mathcal{E}(x):= \pi(x)- \text{li}(x)  \ll x \exp\left( -c\sqrt{\log x}\right)
\end{equation}
for a constant $c>0$, where $\pi(x)$ denotes the prime counting function. (See \cite{TenenbaumBook} for example.)  This implies that for any fixed $A>0$,
\begin{equation} \label{pilibound}
\mathcal{E}(x)\ll_A\frac{x}{\log^A x}.
\end{equation}
Moreover the Riemann hypothesis is equivalent to the bound $\mathcal{E}(x)\ll \sqrt{x}\log x$.  (See for instance \cite[p.\ 70]{lucadekoninck}.)

%Better yet:  FIX THE LOWER BOUND.

%\begin{lem}
%For all $x > 1$,
%\begin{displaymath}
%\pi(x)-\text{li}(x) < \frac{1.59x}{\log^4 x}.
%\end{displaymath}
%\end{lem}
%\begin{proof}
%From \cite[Thm 5.1]{Dusart2018}, we know that
%\begin{displaymath}
%\pi(x) < \frac{x}{\log x} \left( 1 + \frac{1}{\log x} + \frac{2}{\log^2 x} + \frac{7.59}{\log^3 x}\right)
%\end{displaymath}
%for all $x > 1$. As li$(x) > \frac{x}{\log x} \left( 1 + \frac{1}{\log x} + \frac{2}{\log^2 x} + \frac{6}{\log^3 x}\right)$, the result comes from subtraction when $x \ge 2$.  If $1<x \le 2$, then $\pi(x)-\text{li}(x)<0$.
%\end{proof}

We recall the definition of the polylogarithm, 
\begin{equation} \label{polylogdef}
\text{Li}_{j}\left( x\right) := \sum_{n=1}^{\infty} \frac{x^n}{n^j},
\end{equation}
which generalizes $\zeta(j)=\text{Li}_j(1)$. In Appendix \ref{polylogsec}, we prove the following bound which is of use in the proof of Theorem \ref{mainthm}.

\begin{lem} \label{Polylog}
For all $1 < x \le M$ and $j \ge 2$,
\begin{displaymath}
\text{\rm Li}_{1-j} \left( \frac{1}{x} \right) \ll_M \frac{(j-1)!}{(\log x)^j}. 
\end{displaymath}
\end{lem}

In fact, for fixed $x>1$, the upper bound of Lemma \ref{Polylog} is an asymptotic equality. We will give a numerically explicit upper bound for the error term.

We use the following result on the von Mangoldt function $\Lambda(n)$, which is established in \cite{Coffey} via the von Mangoldt explicit formula.

\begin{lem}[Equation (A.2) in \cite{Coffey}] \label{etalemma}
We have
\begin{displaymath}
\lim_{x\to \infty} \left(\frac{\log^j x}{j} - \sum_{n \le x} \Lambda(n) \frac{\log^{j-1}n}{n}\right) = \eta_{j -1}(-1)^j(j-1)!,
\end{displaymath}
where $\eta_j$ are the coefficients in the series
\begin{displaymath}
g(s)=-\frac{\zeta'(s)}{\zeta (s)} -\frac{1}{s -1}=\sum_{j=0}^{\infty} \eta_j (s-1)^j.
\end{displaymath}
\end{lem}

Note that $g$ has an analytic continuation to the open disk of radius 3 centered at $s = 1$, because $\zeta$ has no zeros on this open disk. 

\begin{lem} \label{dklemma}
Let $k \ge 4$, $n \ge 17$, and $d=(n-1)(k-1)$.  Then $k \le d/12$.
\end{lem}
\begin{proof}
First, note that since $d$ is defined as above, $k-1 = d/(n-1) \le d/16.$  Since $d \ge 48$ by the above, $1 \le d/48$, and adding these gives the result.
\end{proof}

% Setting no specific lower limit on $n$, we can use the same argument to show that
% \begin{displaymath}
% k= (k-1) + 1 \le \frac{d}{n-1} + \frac{d}{2(n-1)}=\frac{3d}{2(n-1)}.
% \end{displaymath}
% In the proof above, if $k \ge 2$ instead, the result becomes $d/3$ and $d+\frac{d}{k-1}+k \le 7d/3$. 
% Choosing $n \ge 1+3/(2\epsilon)$, we can improve the result of the Lemma above to $k \le \epsilon d$ for any choice of $\epsilon > 0$.

We will also use a strong form of Mertens' first theorem which follows from the prime number theorem.

\begin{lem} \label{Mertens1}
We have
\begin{displaymath}
\sum_{p \le x} \frac{\log p}{p} = \log x -\alpha_1 + O\left( \frac{1}{\log^2 x} \right)
\end{displaymath}
where $\alpha_1 = \gamma +\sum_p \log p/p(p-1) = 1.332582\ldots$.
\end{lem}

(The constant $\alpha_1$ is also given by \eqref{alphajk} above.) The following proposition gives the asymptotic behavior of $\alpha_{j,k}$ for fixed $k$ as $j\to \infty$.

\begin{prop}\label{alphajkestimate}
Let $k>1$.  For each $m\ge 0$ we have
\begin{displaymath}
\alpha_{j,k} = \frac{(j-1)!}{(k-1)^j}(1+O_{k,m}(j^{-m})).
\end{displaymath}
In particular, for fixed $k>1$ we have $\alpha_{j,k}\sim (j-1)!/(k-1)^j,~~~(j\to\infty)$.
\end{prop}

\begin{proof}
Letting $f(t)=\log^j t/t^k$, we apply partial summation with the prime number theorem (in the form of \eqref{pnt}) and integration by parts to obtain
\begin{equation}\label{errorbound}
\alpha_{j,k} = \lim_{x\to \infty}\sum_{p\le x}\frac{\log^j p}{p^k} = f(2)\text{li}(2)+\int_2^\infty \frac{\log^{j-1}t}{t^k}dt - \int_2^\infty f'(t)\mathcal{E}(t)dt,
\end{equation}
where $\mathcal{E}(t) = \pi(t)-\text{li}(t)$.  The first two terms sum to 
\begin{displaymath}
O(1)+\int_1^\infty \frac{\log^{j-1}t}{t^k}dt,
\end{displaymath}
and this integral evaluates to $(j-1)!/(k-1)^j$ after substituting $u=(k-1)\log t$ and recognizing the integral as that of the gamma function.  It remains to address the integral involving the error term.  This is handled in a similar way using \eqref{pilibound}.  
\end{proof}

\begin{rmk}
For fixed $k>1$ we have $\alpha_{j,k}\sim {\rm Li}_{1-j}(e^{1-k})$ as $j\to\infty$.  (See the appendix.)
\end{rmk}

We will use the following estimate for the sum of reciprocals of squares of primes exceeding a given bound.

\begin{lem}\label{primesquaretail}
Uniformly for $x>1$ and $a \ge 2$, we have
\begin{displaymath}
 \sum_{p > x} \frac{1}{p^a} \ll \frac{1}{x^{a-1}\log x}.
\end{displaymath}
\end{lem}

\begin{proof}
We first prove the case $a=2$, applying partial summation to write
\begin{displaymath}
\sum_{p>x}\frac{1}{p^2} = -\frac{\pi(x)}{x^2}+2\int_x^\infty \frac{\pi(t)}{t^3}dt
\end{displaymath}
and then applying the prime number theorem.  (In fact, Chebyshev's estimate $\pi(t)\ll t/\log t$ is sufficient.) The extension to larger $a$ follows immediately.
\end{proof}

Note that the implied constant in Lemma \ref{primesquaretail} can be taken as 1 by extending the argument for the case $a=2$ due to Nguyen and Pomerance \cite[Lemma 2.7]{nguyen2019reciprocal}.

We have the following additional formula for $\alpha_{j,k}$.

\begin{lem}\label{djk}
For each $k\ge 2$ and $j\ge 0$, we have
\begin{displaymath}
    \alpha_{j,k} = P(k)\log^j 2 + j d_{j,k},
\end{displaymath}
where 
\begin{displaymath}
d_{j,k} := \int_2^\infty \frac{\epsilon_k(t)\log^{j-1}t}{t}dt
\end{displaymath}
and where 
\begin{displaymath}
\epsilon_k(t) := \sum_{p>t}p^{-k}.
\end{displaymath}
\end{lem}

\begin{proof}
By partial summation, we have
\begin{displaymath}
 \alpha_{j,k} = \lim_{x\to \infty}\sum_{p\le x}\frac{\log^j p}{p^k} = \lim_{x\to \infty}\left(s_k(x)\log^j x - j\int_2^x \frac{s_k(t)\log^{j-1}t}{t}~dt\right),
\end{displaymath}
where $s_k(t):=\sum_{p\le t}p^{-k}$.  We then use the identity $s_k(t)=P(k)-\epsilon_k(t)$ and take the limit.  Note that the resulting improper integral is convergent for all $k \ge 2$ because we have $\epsilon_k(t)\ll (t^{k-1}\log t)^{-1}$ for all $k \ge 2$ by Lemma \ref{primesquaretail}.
\end{proof}

\begin{rmk}
In Proposition \ref{alphajkestimate} and Lemma \ref{djk}, $j$ and $k$ may take any real values in the specified ranges upon replacing $(j-1)!$ with $\Gamma(j)$.  However, in our results we will only use integer values of $j$ and $k$.
\end{rmk}

We will also use the Weierstrass product formula for the gamma function $\Gamma(z)$.  (See for instance \cite[Theorem II.0.6]{TenenbaumBook}.)

\begin{lem}\label{weierstrass}
For all $z\in \mathbb{C}$, we have

\begin{equation}
 \frac{1}{\Gamma(z)} = ze^{\gamma z}\prod_{n\ge 1}\left(1+\frac{z}{n}\right)e^{-z/n}.
\end{equation}

\end{lem}

We remark that taking the logarithm of the Weierstrass product formula after applying the functional equation $\Gamma(z+1)=z\Gamma(z)$ gives the expansion

\begin{equation}\label{gammaexpansion}
 \Gamma(z+1) = \exp\left(-\gamma z + \sum_{j\ge 2}\frac{\zeta(j)}{j}(-z)^j\right),~~(|z|<1).
\end{equation}

See \cite[Lemma 5.1]{HME} for instance.  The following lemma gives derivatives of $1/\Gamma$ at the nonpositive integers.

\begin{lem}\label{gammaderivatives}
For all integers $M\ge 1$, we have
\begin{equation}
    \left(\frac{1}{\Gamma}\right)(1-M) = 0,
\end{equation}
\begin{equation} \label{eq17}
    \left(\frac{1}{\Gamma}\right)'(1-M) = (-1)^{M-1}(M-1)!,
\end{equation}
and 
\begin{equation}
    \left(\frac{1}{\Gamma}\right)''(1-M) = 2(-1)^M(M-1)!\left(\sum_{1\le j\le M-1}\frac{1}{j}-\gamma\right).
\end{equation}
\end{lem}

\begin{proof}
The first assertion holds because $\Gamma$ has poles at the nonpositive integers.  For the second and third assertions we use the well-known reflection formula for the gamma function,
\begin{displaymath}
 \left(\frac{1}{\Gamma}\right)(z) = \frac{\sin\pi z}{\pi}\Gamma(1-z),~~(z\in \mathbb{C}\setminus \mathbb{N}).
\end{displaymath}
We differentiate the reflection formula using the product rule.  When $z\le 0$ is an integer, we have $\sin(\pi z) =0$, which simplifies the computations for the derivatives of $1/\Gamma$.  We have
 \begin{displaymath}
  \left(\frac{1}{\Gamma}\right)'(1-M) = (-1)^{M-1}\Gamma(M) = (-1)^{M-1}(M-1)!,~~(M\ge 1),
 \end{displaymath}
 which is the second assertion.  Finally, for the third assertion we have
 \begin{displaymath}
  \left(\frac{1}{\Gamma}\right)''(1-M) = 2(-1)^M\Gamma'(M),~~(M\ge 1).
 \end{displaymath}
 
 Thus it suffices to show that
 \begin{displaymath}
  \Gamma'(M) = (M-1)!\left(\sum_{1\le j\le M-1}\frac{1}{j}-\gamma\right),~~(M\ge 1),
 \end{displaymath}
 
 which we verify by induction on $M$.  For the base step we have $\Gamma'(1)=-\gamma$ by \eqref{gammaexpansion}.  For the induction step we differentiate the functional equation $\Gamma(z+1)=z\Gamma(z)$, obtaining
 
  \begin{displaymath}
  \begin{aligned}
   \Gamma'(M+1) = \Gamma(M)+M\Gamma'
  (M) &= (M-1)!+M(M-1)!\left(\sum_{1\le j\le M-1}\frac{1}{j}-\gamma\right)\\ 
  &= M!\left(\sum_{1\le j\le M}\frac{1}{j}-\gamma\right),~~(M\ge 1).
  \end{aligned}
 \end{displaymath}
 
 This completes the proof of Lemma \ref{gammaderivatives}.
\end{proof}

In principle, the method of the proof of Lemma \ref{gammaderivatives} can be extended recursively to compute $(1/\Gamma)^{(j)}(1-M)$ for any $j\ge 0$ and $M\ge 1$.  Specifically, we take the $j$th derivative of the reflection formula and then take the derivative of order $j-1$ of the functional equation to obtain a recursive formula for $\Gamma^{(j-1)}(M)$.  For example, letting $j=3$ we obtain

\begin{displaymath}
 \left(\frac{1}{\Gamma}\right)'''(1-M) = (-1)^{M-1}(3\Gamma''(M)-\pi^2\Gamma(M)),~~(M\ge 1)
\end{displaymath}
and
\begin{displaymath}
 \Gamma''(M+1) = 2\Gamma'(M)+M\Gamma''(M),~~(M\ge 1).
\end{displaymath}
We have $\Gamma''(1)=\gamma^2+\zeta(2)$ by \eqref{gammaexpansion}.  It follows by induction that
\begin{displaymath}
\begin{aligned}
 \Gamma''(M) &= (M-1)!\left(\left(\sum_{1\le j\le M-1}\frac{1}{j}\right)\left(\sum_{1\le j\le M-1}\frac{1}{j}-2\gamma\right)-\sum_{1\le j\le M-1}\frac{1}{j^2}+\gamma^2+\zeta(2)\right)\\
 & = (M-1)!\left(\left(\sum_{1\le j\le M-1}\frac{1}{j}-\gamma\right)^2+\zeta(2)-\sum_{1\le j\le M-1}\frac{1}{j^2}\right)
 \end{aligned}
\end{displaymath}

We conclude this section with several other formulas for the $\alpha_j$.  We have defined $\alpha_{j,k}$ as a convergent sum over primes, while $\alpha_j$ is defined as a limit involving a divergent sum over primes.  We have the following alternative formula for the $\alpha_j$ which involves a convergent sum over primes, and which generalizes the well-known formula

\begin{equation}\label{alpha1formula}
\alpha_1 = \gamma + \sum_p \frac{\log p}{p(p-1)}
\end{equation}

for the constant appearing in Lemma \ref{Mertens1}. 

\begin{prop}\label{alphajformulaprop}
For each $j\ge 1$, we have
\begin{equation}\label{alphajformula}
\alpha_j = \frac{1}{j}\left(\gamma\log^{j-1}2+\sum_p \frac{\log^j p}{p(p-1)} - \frac{j-1}{j}\log^j 2 + (j-1)\int_2^\infty \frac{\bar{E}(t)\log^{j-2}t}{t}dt\right),
\end{equation}
\end{prop}
where
\begin{displaymath}
\bar{E}(t) := \sum_{p\le t}\frac{\log p}{p-1} - (\log t - \gamma).
\end{displaymath}

\begin{proof}
We apply partial summation with the formula
\begin{equation}\label{logpoverp-1}
    A(t):=\sum_{p\le t}\frac{\log p}{p-1} = \log t - \gamma + \bar{E}(t),
\end{equation} 
where $\bar{E}(t)\ll_A 1/\log^A t$ for all $A>0$ by a strong form of the prime number theorem.   (See for instance \cite[p.\ 25]{TenenbaumBook}.) We thus have
\begin{displaymath}
\sum_{p\le x}\frac{\log^j p}{p} = \sum_{p\le x}\frac{\log^j p}{p-1} - \sum_{p\le x}\frac{\log^j p}{p(p-1)}
\end{displaymath}
and
\begin{displaymath}
\sum_{p\le x}\frac{\log^j p}{p-1} = A(x)\log^{j-1}x - (j-1) \int_2^x \frac{A(t)\log^{j-2}t}{t}dt.
\end{displaymath}
We apply formula \eqref{logpoverp-1} and then take the limit as $x\to \infty$.
\end{proof}

\begin{rmk}
Note that when $j=1$, Proposition \ref{alphajformulaprop} reduces to \eqref{alpha1formula}.  Note also that 
\begin{displaymath}
\sum_p \frac{\log^j p}{p(p-1)} = \sum_{k\ge 2}\alpha_{j,k}
\end{displaymath}
and in particular
\begin{displaymath}
\alpha_1 = \gamma + \sum_{k\ge 2}\alpha_{1,k}.
\end{displaymath}
\end{rmk}

We will show in \eqref{etaalpha} in the following section that

\begin{displaymath}
\alpha_j = \frac{1}{j}\left(\sum_{k\ge 2}k^{j-1}\alpha_{j,k}+(-1)^j (j-1)!\eta_{j-1}\right) \sim \frac{2^{j-1}\alpha_{j,2}}{j},~~~(j\ge 1).
\end{displaymath}
Solving for $\eta_j$, we have
\begin{displaymath}
\eta_j = \frac{(-1)^j}{j!}\sum_{k\ge 1} k^j\alpha_{j+1,k},~~~~(j\ge 0).
\end{displaymath}

We note that the $\eta_j$ have a significant role in other results.  In particular, Murty and Pathak \cite[p.\ 5-6]{murtyrelations} proved that the Riemann hypothesis can be deduced from properties of the $\eta_j$.

\section{Proof of Theorem \ref{mainthm}}\label{mainthmsec}

In this section we establish Theorem \ref{mainthm}. The proof will proceed by passing from sums over primes $p$ to sums over all prime powers $p^k$, to which Lemma \ref{etalemma} applies. The main term in the resulting estimates for $\alpha_j$ arises from the prime squares $p^2$.

Now recall that by \cite[Lemma 8.3]{HME}, we have
\begin{align}
\alpha_j & = \lim_{x\to\infty}\frac{1}{j}\left(\frac{(\log x)^j}{j} - \sum_{p\le x}\frac{(\log p)^j}{p}\right).
\end{align}
We note
\begin{align*}
\sum_{n\le x}\frac{\Lambda(n)}{n}(\log n)^{j-1} \ = \ \sum_{p\le x}\frac{(\log p)^j}{p} + \sum_{k\ge2}\sum_{p^k\le x}\frac{(\log p)^j}{p^k}k^{j-1}
\end{align*}
so that
\begin{displaymath}
\alpha_j = \frac{1}{j}\cdot \lim_{x \to \infty} \left(\frac{(\log x)^j}{j} - \sum_{n \le x} \Lambda(n) \frac{(\log n)^{j-1}}{n} + \sum_{k \ge 2} k^{j-1} \sum_{p^k \le x} \frac{(\log p)^j}{p^k}\right),
\end{displaymath}
which by Lemma \ref{etalemma} gives
\begin{align}\label{etaalpha}
\alpha_j = \frac{1}{j} \left( \sum_{k \ge 2} k^{j-1} \sum_{p}\frac{(\log p)^j}{p^k} \ + \ (-1)^j(j-1)! \eta_{j -1}\right).
\end{align}

Note the coefficients $\eta_j$ are are bounded as $\eta_j \ll_{\epsilon} (3 - \epsilon)^ {-j}$ for any $0 <\epsilon<3$, since by Lemma \ref{etalemma}, $g(s)$ may be analytically continued on the open disk of radius 3 centered at $s=1$.  Therefore $(-1)^j(j-1)! \eta_{j -1}\ll_{\epsilon} (j-1)!(3 - \epsilon)^ {-j}$ so the contribution from this term is negligible.

In summary, $\alpha_j = (S+(-1)^j(j-1)!\eta_{j -1})/j$ for the series $S$,
\begin{align*}
S:= \sum_{k \ge 2} k^{j-1} \sum_{p } \frac{(\log p)^j}{p^k} = \sum_{k \ge 2} k^{j-1} \alpha_{j,k} =  S_1 + S_2 + S_3 + S_4,
\end{align*}
decomposing $S$ as
\begin{align*}
S_1 & = \sum_{k \ge 4} k^{j-1} \sum_{p\le e^{16}} \frac{(\log p)^j}{p^k}, &
S_2 & = \ 2^{j-1} \sum_{p}\frac{(\log p)^j}{p^2},\\
S_4 & =\sum_{k \ge 4} k^{j-1} \sum_{p> e^{16}} \frac{(\log p)^j}{p^k}, &
S_3 & =\ 3^{j-1} \sum_{p}\frac{(\log p)^j}{p^3}.
\end{align*}

First we observe that $S_1 \ll (j-1)!$, since for each prime $p \le e^{16}$, Lemma \ref{Polylog} gives
$$\sum_{k\geq 4} k^{j-1} (\log p)^j/p^k<(\log p)^j\text{Li}_{1-j}(1/p) \ll (j-1)!.$$

For $S_3$, by Proposition \ref{alphajkestimate} we have
\begin{align}
S_3=3^{j-1}\alpha_{j,3} = 3^{j-1}\frac{(j-1)!}{2^j}(1+o(1)) = 2^{j-1}(j-1)!(1+o(1))\frac{2}{3}\left(\frac{3}{4}\right)^j.
\end{align}

For $S_4$, by dyadic decomposition and Lemma \ref{Mertens1} we have
\begin{align}\label{k>3sum}
S_4 & = \sum_{k\geq 4} \sum_{p>e^{16}}\frac{(k\log p)^j}{kp^k} \ \le \ \sum_{k\geq 4}\sum_{n \ge 17}\frac{(kn)^{j-1}}{e^{(n-1)(k-1)}}\sum_{e^{n-1} < p \le e^n}\frac{\log p}{p} \nonumber\\
& \ \ll \sum_{k\geq 4}\sum_{n \ge 17}\frac{(kn)^{j-1}}{e^{(n-1)(k-1)}} = \sum_{\substack{k\geq 4,n \ge 17\\ d:=(n-1)(k-1)}} \left(k+d+\frac{d}{k-1}\right)^{j-1}e^{-d}.
\end{align}
Now $4\le k\le d/12$ by Lemma \ref{dklemma},  and so 
\begin{displaymath}
S_4\ll \sum_{d \ge 48} \tau(d)\left(\frac{17d}{12}\right)^{j-1}e^{-d} \ll_{\epsilon} \left( \frac{17}{12} \right)^{j-1} \sum_{d \ge 48} \frac{d^{j-1}}{(e-\epsilon)^d}
\end{displaymath}
for any $0<\epsilon<e$. Here $\tau$ denotes the divisor function, which satisfies $\tau(d)\le 2\sqrt{d}$. Thus by Lemma \ref{Polylog} with $x=e-\epsilon$, for $\epsilon>0$ sufficiently small,
\begin{align}\label{eq:S4}
S_4 = o\left( (3/2)^{j-1} (j-1)!\right).
\end{align}

It remains to address $S_2$.  By Proposition \ref{alphajkestimate}, we have
\begin{displaymath}
 S_2 = 2^{j-1}(j-1)!(1+O_m(j^{-m}))
\end{displaymath}
for any $m\ge 0$. Hence collecting the above estimates, we conclude
\begin{align}
\alpha_j = \frac{1}{j}\big(S_1+S_2+S_3+S_4+(-1)^j (j-1)!\eta_{j -1}\big) = \frac{2^{j-1}(j-1)!}{j} \left( 1 + O_m\left( j^{-m}\right) \right). 
\end{align}
This gives the unconditional part of Theorem \ref{mainthm}.

To complete the proof, we now assume the Riemann hypothesis.
It remains to prove the conditional part of Theorem \ref{mainthm}.  Let $b(x):=2^{j-1}(j-1)!\frac{2}{3}\left(\frac{3}{4}\right)^j$.  By the work above, we have
\begin{displaymath}
\begin{aligned}
    &S_1 = o(b(x)),&\qquad
    &S_3 = (1+o(1))b(x),\\
    &S_4 = o(b(x)),&\qquad
    &\eta_{j-1}(j-1)!=o(b(x)).
\end{aligned}
\end{displaymath}

Also, by equation \eqref{errorbound}, we have
\begin{align}
 S_2 = 2^{j-1}((j-1)!+O(1) - J),
\end{align}
where
\begin{displaymath}
 J:= \int_2^\infty \mathcal{E}(t) \cdot t^{-3} (\log t)^{j-1} (j - 2\log t) dt.
\end{displaymath}
Assuming the Riemann hypothesis we have $\mathcal{E}(t)\ll \sqrt{t}\log t$, in which case
\begin{displaymath}
J \ll j \int_1^{\infty} t^{-5/2} (\log t)^{j} dt + \int_1^{\infty} t^{-5/2} (\log t)^{j+1} dt.
\end{displaymath}
Substituting with $s=3\log t/2$, we therefore have
\begin{displaymath}
\begin{aligned}
    J & \ll j\int_0^{\infty} e^{-s} \left(\frac{2s}{3}\right)^{j} ds + \int_0^{\infty} e^{-s} \left(\frac{2s}{3}\right)^{j+1} ds\\
    &= \left(\frac{2}{3}\right)^{j} \left( j \Gamma(j+1) + \frac{2}{3} \Gamma(j+2)\right) 
    \ll \left(\frac{2}{3}\right)^{j}(j+1)!
    \ll j^2\left(\frac{2}{3}\right)^{j}(j-1)!.
\end{aligned}
\end{displaymath}
Thus $S_2 =2^{j-1}(j-1)![1+O(j^2\,(2/3)^{j})]$. Hence assuming the Riemann hypothesis we conclude
\begin{align}\label{eq:alphajRH}
\alpha_j = \frac{2^{j-1}(j-1)!}{j}\left(1+\left(3/4\right)^j(2/3+o(1))\right).
\end{align}
This completes the proof of Theorem \ref{mainthm}.

\section{Fine-scale asymptotics for $\mathcal{S}_k(x)$}\label{notes}

Following the method in Tenenbaum \cite{tenenbaum2016generalized}, we show that the constants $\alpha_j$ arise in the fine-scale asymptotic of $\mathcal{S}_k(x)$ for each $k\ge 2$.  The following proposition addresses the cases $k=2,3$. In principle, with great patience, the proof method can be used to address $\mathcal{S}_k(x)$ for arbitrary $k$.

\begin{prop}\label{s2s3}
For any integer $N>0$, we have
\begin{equation}\label{s2}
\mathcal{S}_2(x) = (\log_2 x+\beta)^2-\zeta(2)+\sum_{1\le j<N}\frac{2\alpha_j}{\log^j x}+O_N\left(\frac{\log_2 x}{\log^{N}x}\right).
\end{equation}

Also, for any integer $N>0$, we have
\begin{equation}\label{s3}
\begin{aligned}
\mathcal{S}_3(x) &= (\log_2 x+\beta)^3-3\zeta(2)(\log_2 x+\beta)+2\zeta(3)\\
&\quad +\sum_{1\le k<N}\frac{6\alpha_k(\log_2 x+\beta)-v_k}{\log^k x} 
\ + \ O_N\left(\frac{\log_2 x}{\log^{N}x}\right),
\end{aligned}
\end{equation}
for constants $v_k$ given by
\begin{equation}\label{vn}
v_k = 3\sum_{1\le j\le k-1}\left(\frac{2\alpha_k}{j}+j\binom{k-1}{j}\alpha_j\alpha_{k-j}\right).
\end{equation}
In particular $v_1=0$, $v_2=6\alpha_2+3\alpha_1^2$, and $v_3=9\alpha_3+12\alpha_1\alpha_2$. Moreover, $v_k$ satisfies
\begin{equation}\label{vnestimate}
    v_k = 9\alpha_k(\log k + O(1)).
\end{equation}
\end{prop}

Note that the estimate \eqref{s2} for $\mathcal{S}_2(x)$ is equivalent to the fine-scale estimate for $\mathcal{R}_2(x)$ in \cite[Theorem 1.3]{HME}.  (The error term above can be reduced to $O_N(1/\log^{N}x)$ by adding an extra term to the sum corresponding to $j=N$.)

Before proving Proposition \ref{s2s3}, we will state and prove a preliminary result.  Denote
\begin{displaymath}
 h(s):= P(s+1)-\log(1/s),\quad (\text{Re}(s)>0).
\end{displaymath}
As observed in \cite[p.\ 2]{tenenbaum2016generalized}, $h$ has an analytic continuation to a neighborhood of $s=0$.  Recalling the notation $B_n=-n\alpha_n$, we have the following lemma.

\begin{lem}\label{hprop}
We have $h(0)=\beta-\gamma$, and for each $n\ge 1$ we have
\begin{displaymath}
 h^{(n)}(0) = (-1)^n B_n = (-1)^{n-1}n\alpha_n.
\end{displaymath}

\end{lem}

\begin{proof}
The first assertion is due to Tenenbaum \cite[p.\ 2]{tenenbaum2016generalized}. For the second assertion, we follow Cri{\c{s}}an--Erban \cite[p.\ 10-11]{cricsan2021counting} in the computation of the constants $B_n$.  Specifically, note by M\"obius inversion,
\begin{displaymath}
P(s) = \log \zeta(s) + \sum_{j\ge 2}\frac{\mu(j)}{j}\log \zeta(js),\quad (\text{Re}(s)>1).
\end{displaymath}
Differentiating $n$ times, we have
\begin{align}
P^{(n)}(s)=\left(\frac{\zeta'}{\zeta}\right)^{(n-1)}(s)+ \sum_{j\ge 2}\mu(j)j^{n-1}\left(\frac{\zeta'}{\zeta}\right)^{(n-1)}(js).
\end{align}
Now for $\text{Re}(s)>1$, we let $H(s):= h(s-1) = P(s) - \log(1/(s-1))$.  We thus have
\begin{align}
H^{(n)}(s) = P^{(n)}(s) + \left(\frac{1}{s-1}\right)^{(n-1)} = \ \sum_{j\ge 2}\mu(j)j^{n-1}\left(\frac{\zeta'}{\zeta}\right)^{(n-1)}(js) + \left(\frac{\zeta'(s)}{\zeta(s)}+\frac{1}{s-1}\right)^{(n-1)}.
\end{align}
Now $H$ has an analytic continuation to a neighborhood of $1$, and letting $s\to 1$, we obtain
\begin{displaymath}
h^{(n)}(0) = H^{(n)}(1) = \sum_{j\ge 2}\mu(j)j^{n-1}\left(\frac{\zeta'}{\zeta}\right)^{(n-1)}(j) + \lim_{s\to 1}\left(\frac{\zeta'(s)}{\zeta(s)}+\frac{1}{s-1}\right)^{(n-1)}.
\end{displaymath}
By equation (33) of \cite{cricsan2021counting}, the right expression above is equal to $(-1)^n B_n$.  Recalling $B_n = -n\alpha_n$ completes the proof.
\end{proof}

We are now in a position to prove Proposition \ref{s2s3}.

\begin{proof}[Proof of Proposition \ref{s2s3}]

As in Tenenbaum \cite[p.\ 2]{tenenbaum2016generalized}, for each integer $N>0$ we have
\begin{align}
\mathcal{S}_k(x) = \frac{1}{2\pi i}\int_{\mathcal{H}}\left(\log\frac{1}{s}+\sum_{0\le j\le N}\frac{h^{j}(0)}{j!}s^{j}\right)^k x^s\frac{ds}{s} + O_N\left(\frac{(\log_2 x)^{k-1}}{(\log x)^{N+1}}\right),
\end{align}

where $\mathcal{H}$ is a Hankel contour around $(-\infty,0]$, oriented positively.  By the multinomial theorem,
\begin{displaymath}
\left(\log\frac{1}{s}+\sum_{0\le j\le N}\frac{h^{j}(0)}{j!}s^{j}\right)^k = \mathop{\sum_{k=m+k_0+\cdots+k_{N}}}_{m,k_0,\ldots,k_{N}\ge 0}\left(\log\frac{1}{s}\right)^m \frac{h(0)^{k_0}}{k_0!}\cdots \frac{\big(h^{(N)}(0)s^{N}/N!\big)^{k_{N-1}}}{k_{N}!}\frac{k!}{m!}.
\end{displaymath}
Thus we have
\begin{align}
\mathcal{S}_k(x) &=  O_N\left(\frac{(\log_2 x)^{k-1}}{(\log x)^{N+1}}\right) \ + \\
&\mathop{\sum_{k=m+k_0+\cdots+k_{N}}}_{m,k_0,\ldots,k_{N}\ge 0}\frac{h(0)^{k_0}}{k_0!}\cdots \frac{\big(h^{(N)}(0)/N!\big)^{k_{N}}}{k_{N}!} \frac{k!}{m!} I(m,k_1+2k_2+\cdots +Nk_{N},x)\nonumber
\end{align}
where
\begin{displaymath}
I(m,M,x) := \frac{1}{2\pi i}\int_{\mathcal{H}} \left(\log\frac{1}{s}\right)^m s^M x^s\frac{ds}{s}.
\end{displaymath}

Now consider Hankel's formula for the gamma function (\cite[p.\ 2]{tenenbaum2016generalized}),
\begin{align}\label{eq:HankelGamma}
\frac{1}{2\pi i}\int_{\mathcal{H}} \frac{x^s}{s^{1+z}} ds = \frac{(\log x)^z}{\Gamma(z+1)}.
\end{align}

Differentiating $m$ times with respect to $z$ and evaluating at $z=-M$, we deduce
\begin{align}\label{imnx}
I(m,M,x) = \frac{1}{2\pi i} \int_{\mathcal{H}} \left(\log\frac{1}{s}\right)^m s^M x^s\frac{ds}{s} 
&= \frac{d^m}{dz^m} \frac{\left(\log x\right)^z}{\Gamma(z+1)} \Big|_{z=-M} \nonumber \\
&= \sum_{j=0}^m \binom{m}{j} \frac{(\log_2 x)^j}{(\log x)^M}\left(\frac{1}{\Gamma}\right)^{(m-j)}(1-M).
\end{align}
Hence letting $M:=k_1+2k_2+\cdots + Nk_{N}$, we obtain
\begin{align}\label{explicitsk}
\mathcal{S}_k(x) &=  O_N\left(\frac{(\log_2 x)^{k-1}}{(\log x)^{N+1}}\right) \ + \\
& \ \mathop{\sum_{k=m+k_0+\cdots+k_{N}}}_{m,k_0,\ldots,k_{N}\ge 0}\frac{h(0)^{k_0}}{k_0!}\cdots \frac{\big(h^{(N)}(0)/N!\big)^{k_{N}}}{k_{N}!} \sum_{j=0}^m \frac{k!}{j!(m-j)!}\frac{(\log_2 x)^j}{(\log x)^M}\left(\frac{1}{\Gamma}\right)^{(m-j)}(1-M).\nonumber
\end{align}

Note that we need only consider the terms with $M=k_1+2k_2+\ldots+ Nk_{N}\le N$, since the remaining terms are absorbed by the error term.  By Lemma \ref{hprop}, we have $h(0)=\beta-\gamma$ and $h^{(n)}(0) = (-1)^{n-1}n\alpha_n$ for $n\ge 1$.  

Furthermore, the derivatives of $1/\Gamma$ at $1-M$ can be computed using the technique of Lemma \ref{gammaderivatives}.  We have the following useful simplifications.  Note that by Lemma \ref{gammaderivatives}, we have

\begin{equation}\label{i0bx}
    I(0,M,x)=0,~~(M\ge 1)
\end{equation} 

and 

\begin{equation}\label{i1bx}
    I(1,M,x)=\frac{(-1)^{M-1}(M-1)!}{\log^M x},~~(M\ge 1).
\end{equation}  

Now consider $k=2$. In this case, aside from the choice $m=0, k_0=2$, by \eqref{i0bx} we require $m\ge1$, so the only additional contributing partitions of 2 are given by $m=2$, or by $m=1$ and $k_j=1$ for some unique $j < N$.
Thus \eqref{explicitsk} simplifies as
\begin{align*}
&\mathcal{S}_2(x) +  O_N\left(\frac{\log_2 x}{(\log x)^{N+1}}\right)\\ 
& = \mathop{\sum_{2=m+k_0+\cdots+k_{N}}}_{m,k_0,\ldots,k_{N}\ge 0}\frac{h(0)^{k_0}}{k_0!}\cdots \frac{\big(h^{(N)}(0)/(N-1)!\big)^{k_{N}}}{k_{N}!}\frac{2}{m!} I(m,M,x) \\
&= I(2,0,x) + h(0)^2I(0,0,x)\ + 2
\mathop{\sum_{1=k_0+\cdots+k_{N}}}_{m,k_0,\ldots,k_{N}\ge 0}\frac{h(0)^{k_0}}{k_0!}\cdots \frac{\big(h^{(N)}(0)/N!\big)^{k_{N}}}{k_{N}!} I(1,M,x) \\
&= I(2,0,x) + h(0)^2I(0,0,x) + 2\sum_{0\le j \le N}\big(h^{(j)}(0)/j!\big) I(1,j,x)
\end{align*}
where $M=k_1+2k_2+\ldots+Nk_{N}$. Recalling Lemma \ref{hprop} and \eqref{i1bx}, we have
\begin{align*}
&\mathcal{S}_2(x) +  O_N\left(\frac{\log_2 x}{(\log x)^{N+1}}\right)\\
&=I(2,0,x) + h(0)^2I(0,0,x) + 2h(0)I(1,0,x) + 2\sum_{1\le j\le N}\frac{(-1)^{j-1}j\alpha_j}{j!}\frac{(-1)^{j-1}(j-1)!}{\log^j x}\\
&= (\log_2 x)^2 +2\gamma\log_2 x + \gamma^2 - \zeta(2) + (\beta-\gamma)^2 + 2(\beta-\gamma)(\log_2 x + \gamma) +\sum_{1\le j\le  N}\frac{2\alpha_j}{\log^j x}\\
&= (\log_2 x + \beta)^2 - \zeta(2) + \sum_{1\le j\le N}\frac{2\alpha_j}{\log^j x}.
\end{align*}
Here we also used \eqref{gammaexpansion} to compute the derivatives of $1/\Gamma$ at $1$.

Now consider $k=3$. In this case \eqref{explicitsk} simplifies as
\begin{align*}
\mathcal{S}_3(x) &=  O_N\left(\frac{(\log_2 x)^2}{(\log x)^{N+1}}\right) \ + \\
&\mathop{\sum_{3=m+k_0+\cdots+k_{N}}}_{m,k_0,\ldots,k_{N}\ge 0}\frac{h(0)^{k_0}}{k_0!}\cdots \frac{\big(h^{(N)}(0)/N!\big)^{k_{N}}}{k_{N}!} \frac{3!}{m!}I(m,M,x).\nonumber
\end{align*}

By \eqref{i0bx}, when $M\ge 1$ we only need to consider terms $I(m,M,x)$ with $m\ge 1$.  We therefore have the following contributing partitions of $3$, arranged in order of increasing $M$:

\begin{enumerate}
    \item $m=3, k_j=0$ for all $j\ge 0$, $M=0$.
    \item $m=2, k_0=1, k_j=0$ for all $j\ge 1$, $M=0$.
    \item $m=1, k_0=2, k_j=0$ for all $j\ge 1$, $M=0$.
    \item $m=0, k_0=3, k_j=0$ for all $j\ge 1$, $M=0$.
    \item $m=2, k_j=1$ for some $j\ge 1$, $k_i=0$ for all $i\ne j$, $M=j$.\label{k=3case5}
    \item $m=1, k_0=1, k_j=1$ for some $j\ge 1$, $k_i=0$ if $i\ge 1$ and $i\ne j$, $M=j$.\label{k=3case6}
    \item $m=1$, $k_j=2$ for some $j\ge 1$, $k_i=0$ for all $i\ne j$, $M=2j$.\label{k=3case7}
    \item $m=1, k_i=k_j=1$ for some pair $i,j\ge 1$ where $i\ne j$, $k_{\ell}=0$ if $\ell\ne i$ and $\ell\ne j$, $M=i+j$.\label{k=3case8}
\end{enumerate}

The terms $M=0$ are addressed by \cite{tenenbaum2016generalized}, so we consider the last four cases where $M\ge 1$.  Note that the second-to-last case only occurs when $2|M$.  For fixed $M\ge 1$, we therefore find that the coefficient of the term $1/\log^M x$ in the expansion of $\mathcal{S}_3(x)$ is given by

\begin{equation}\label{k=3}
\begin{aligned}
 \frac{3h^{(M)}(0)}{M!}I(2,M,x)+\frac{6h(0)h^{(M)}(0)}{M!}I(1,M,x)+3I(1,M,x)\sum_{1\le j\le M-1}\frac{h^{(j)}(0)}{j!}\frac{h^{(M-j)}(0)}{(M-j)!}.
 \end{aligned}
\end{equation}
Specifically, the first, second, and third summands in \eqref{k=3} come from items \ref{k=3case5}, \ref{k=3case6}, and \ref{k=3case7}-\ref{k=3case8}, respectively.  By Lemma \ref{hprop}, the expression in \eqref{k=3} is equal to

\begin{displaymath}
\begin{aligned}
 \frac{3(-1)^{M-1}\alpha_M}{(M-1)!}I(2,M,x)&+\frac{6(\beta-\gamma)(-1)^{M-1}\alpha_M}{(M-1)!}I(1,M,x)\\
 &\quad +3I(1,M,x)\sum_{1\le j\le M-1}\frac{(-1)^{j-1}\alpha_j}{(j-1)!}\frac{(-1)^{M-1-j}\alpha_{M-j}}{(M-1-j)!}.
 \end{aligned}
 \end{displaymath}
By \eqref{i1bx} we have $I(1,M,x)=(-1)^{M-1}(M-1)!/\log^M x$.  By Lemma \ref{gammaderivatives} and \eqref{imnx}, we have
\begin{displaymath}
 I(2,M,x) = \frac{2(-1)^M (M-1)!}{\log^M x}\left(\sum_{1\le j\le M-1}\frac{1}{j} - \gamma - \log_2 x\right).
\end{displaymath}
Combining these results, we obtain the desired estimate \eqref{s3} for $\mathcal{S}_3(x)$, noting that the error term can be reduced to $O_N\big((\log_2 x)/(\log x)^{N+1}\big)$ by adding an extra term to the sum corresponding to $j=N$.

It remains to prove \eqref{vnestimate}. For this we use the estimate
\begin{equation}\label{harmonicsum}
    \sum_{1\le j\le N-1}1/j = \sum_{1\le j\le N-1}1/(N-j)=\log N + O(1).
\end{equation}
From \eqref{vn}, we have
\begin{align}
v_N &= 6\alpha_N\sum_{1\le j\le N-1}\frac{1}{j}+3\sum_{1\le j\le N-1}j\binom{N-1}{j}\alpha_j\alpha_{N-j}.
\end{align}
We use \eqref{harmonicsum} for the first sum above. By Theorem \ref{mainthm} we have $\alpha_j = \big(1+O(j^{-2})\big)(j-1)!2^{j-1}/j$ so the second sum above is
\begin{align*}
v_N - 6\alpha_N(\log N+O(1)) &= 3\sum_{1\le j\le N-1}j\binom{N-1}{j}\alpha_j\alpha_{N-j} \\
&= 3\cdot 2^{N-2}(N-1)!\sum_{1\le j\le N-1}\frac{1}{j(N-j)}\left(1+O(j^{-2})\right)\left(1+O((N-j)^{-2})\right)\\
&=3\frac{2^{N-2}(N-1)!}{N}\sum_{1\le j\le N-1}\left(\frac{1}{j}+\frac{1}{N-j}\right)\left(1+O(\min(j,N-j)^{-2})\right)\\
&=3\frac{2^{N-1}(N-1)!}{N}(\log N+O(1))\\
&= 3(\alpha_N+O(N^{-2}))(\log N+O(1))
\end{align*}
again by Theorem \ref{mainthm}. Hence $v_N = 9\alpha_N(\log N+O(1))$ as desired. This completes the proof of estimate \eqref{vnestimate} and of Proposition \ref{s2s3}.

\iffalse
Note that the power of $\log_2 x$ is reduced to $1$ in the numerator of \eqref{s3} due to the poles of $\Gamma$ corresponding to zeros of $1/\Gamma$.  To determine the coefficient of $\log_2 x$ as $6\alpha_M$, we apply the multinomial theorem for the cube of a sum, Lemma \ref{hprop}, and \eqref{i1bx}, obtaining 
\begin{displaymath}
\frac{3(-1)^{M-1}M\alpha_M}{M!}\frac{2(\log_2 x)}{\log^M x}\left(\frac{1}{\Gamma}\right)'(1-M) = \frac{6\alpha_M\log_2 x}{\log^M x}.
\end{displaymath}

For the formula \eqref{vn} for the coefficients $v_M$, we also use the identity
\begin{displaymath}
\left(\frac{1}{\Gamma}\right)''(1-M) = 2(-1)^M(M-1)!\left(\sum_{1\le j\le M-1}\frac{1}{j}-\gamma\right)
\end{displaymath}
which follows from the Weierstrass product formula.
\fi

\end{proof}

Using the same methods as those in the proof of Proposition \ref{s2s3}, we may also consider the coefficient of $1/\log x$ in the $\mathcal{S}_k(x)$ expansion for arbitrary $k\ge 1$.  For $k=1$, the $1/\log x$ term in the $\mathcal{S}_1(x)$ expansion is $0$, while for $k=2$, the $1/\log x$ term in the $\mathcal{S}_2(x)$ expansion is 
\begin{displaymath}
\frac{2\alpha_1}{\log x}.
\end{displaymath}

For $k=3$, the $1/\log x$ term in the $\mathcal{S}_3(x)$ expansion is 
\begin{displaymath}
\frac{6\alpha_1(\log_2 x+\beta)}{\log x}.
\end{displaymath}

Regarding the case $k=4$, note that the $1/\log x$ term in the $\mathcal{S}_4(x)$ expansion is 

\begin{equation}\label{s41}
\frac{12\alpha_1((\log_2 x+\beta)^2-\zeta(2))}{\log x}.
\end{equation}

Similarly, regarding the case $k=5$, the $1/\log x$ term in the $\mathcal{S}_5(x)$ expansion is 

\begin{displaymath}
\frac{20\alpha_1((\log_2 x+\beta)^3-3\zeta(2)(\log_2 x+\beta)+2\zeta(3))}{\log x}.
\end{displaymath}

Thus for all $k\le 5$, the $1/\log x$ term in the $\mathcal{S}_k(x)$ expansion is
\begin{align}
\frac{k(k-1)\alpha_1S_{k-2}(\log_2 x)}{\log x},
\end{align}
perhaps suggesting that this formula holds in general.

We also establish the following result regarding $\mathcal{S}_4(x)$.

\begin{prop}
The $1/\log^2 x$ term in the fine-scale expansion of $\mathcal{S}_4(x)$ is given by
\begin{equation}\label{s42}
\frac{12\alpha_2((\log_2 x+\beta)^2-\zeta(2))-12(\alpha_1^2+2\alpha_2)(\log_2 x+\beta)+12\alpha_1^2}{\log^2 x}.
\end{equation}
\end{prop}

\begin{proof} 
We follow the proof of Proposition \ref{s2s3} above.  When $k=4$ the integrand in the first displayed equation is 
\begin{displaymath}
 \left(\log\frac{1}{s}+\sum_{0\le j<N}\frac{h^{j}(0)}{j!}s^{j}\right)^4 \frac{x^s}{s}.
\end{displaymath}
We apply the multinomial theorem, collecting all terms exactly divisible by $s^2$.  Recalling that $I(0,N,x)=0$ for all integers $N\ge 1$, and using Lemma \ref{hprop}, we find that the $1/\log^2 x$ term of $\mathcal{S}_4(x)$ simplifies to
\begin{displaymath}
\begin{aligned}
 &-4\alpha_2 I(3,2,x) + (6\alpha_1^2 -12 \alpha_2(\beta-\gamma))I(2,2,x)\\
 &\qquad +12((\beta-\gamma)\alpha_1^2 - \alpha_2(\beta-\gamma)^2)I(1,2,x).
 \end{aligned}
\end{displaymath}
We then use \eqref{imnx} with $M=2$.  Specifically, we can evaluate the derivatives of $1/\Gamma$ at $1-M=-1$ by using \eqref{gammaexpansion} and the techniques of Lemma \ref{gammaderivatives}.  We obtain
\begin{displaymath}
\begin{aligned}
 &I(1,2,x) = -1/\log^2 x,\\
 &I(2,2,x) = - 2(\log_2 x + \gamma -1)/\log^2 x,\\
 &I(3,2,x) = -3((\log_2 x)^2 + (2\gamma-2)\log_2 x - 2\gamma +\gamma^2 - \zeta(2))/\log^2 x.
\end{aligned}
\end{displaymath}
We complete the proof by combining these results.
\end{proof}

We also have the following improvement for the asymptotic estimate for $\mathcal{S}_k$, which applies to $\mathcal{R}_k$ as well (upon dividing by $k!$).

\begin{prop}
For all $k \ge 1$ and $M \ge 1$, the leading term of $\mathcal{S}_k(x)$ with denominator $\log^M x$ is given by
\begin{displaymath}
 \frac{k(k-1)\alpha_M (\log_2 x)^{k-2}}{\log^M x}.
\end{displaymath}
Therefore, the general term with denominator $\log^M x$ is given by
\begin{displaymath}
 \frac{k(k-1)\alpha_M (\log_2 x)^{k-2}(1+O(1/\log_2 x))}{\log^M x}.
\end{displaymath}
\end{prop}

\begin{proof}
This is already established for $k=1,2$, so we assume $k \ge 3$.  Following \eqref{explicitsk}, we see that $m=k$ forces $M=0$, so this case can be ruled out and we must have $m<k$.  For $m=k-1$, the only term in the innermost sum that will produce $(\log_2 x)^{k-1}$ is for $j=m=k-1$, and in this case the coefficient is given by $1/\Gamma(1-M)=0$ due to the poles of $\Gamma$ at the non-positive integers.  Thus, the highest power of $\log_2 x$ appearing in such a term must be at most $k-2$.

Such a term can only appear when $m\ge k-2$ and in that case comes from the $j=k-2$ term in the innermost sum in \eqref{explicitsk}.  By the same argument about the poles of $\Gamma$, $m \neq k-2$, and so we must have $m=k-1$.  This means that all but one of the $k_i$'s must be 0, while the lone exception is 1.  Moreover, $M\ge 1$ implies that this lone exception cannot be $k_0$.  Let $k_i=1$ be this exception.

Now, since we have $m=k-1$ and $j=k-2$, we can use Lemma \ref{hprop} and \eqref{eq17} to see that this term's coefficient is
\begin{displaymath}
 \frac{h^{(i)}(0)}{i!} k(k-1) \frac{(\log_2 x)^{k-2}}{\log^M x} (-1)^{M-1} (M-1)!.
\end{displaymath}
Since $M=i$, this gives the result.
\end{proof}

We now consider the next term.

\begin{prop}\label{k-3}
For all $M\ge 1$, the coefficient of $(\log_2 x)^{k-3}/\log^M x$ in $\mathcal{S}_k(x)$ is given by

\begin{displaymath}
k(k-1)(k-2)\bigg(\alpha_M\big(\beta-h_1(M-1)\big)-\frac{1}{2}\sum_{1\le i\le M-1}i\binom{M-1}{i}\alpha_i\alpha_{M-i}\bigg)
\end{displaymath}
where $h_1(n):=\sum_{1\le i\le n}1/i$.
\end{prop}

\begin{proof}
Let $M\ge 1$.  We apply \eqref{explicitsk} with $N=M$, so that $k_1+2k_2+\cdots +Mk_M=M$.  We sum all terms such that $j=k-3$.  We must have $m\ge k-3$ in \eqref{explicitsk}, and in fact $m\ge k-2$.  Indeed, when $m=k-3$ we have $m-j=0$ and the contribution from this term is $0$ due to the poles of $\Gamma$.  Also, $m\le k-1$, because $m=k$ implies $M=0$, a contradiction.  This leaves only two cases: $m=k-1$ and $m=k-2$.
\begin{enumerate}
    \item Case I: $m=k-1$.  In this case $k_0+\ldots+k_M=1$.  Also $k_0=0$ because $k_1+2k_2+\ldots+Mk_M=M\ge 1$.  Thus $k_i=1$ for a unique $i\ge 1$, and $k_j=0$ otherwise.  Thus $M=i$ and we obtain from (40) and Lemma \ref{hprop} the expression
    \begin{displaymath}
    \frac{(-1)^{i-1}\alpha_i}{(i-1)!}\frac{k(k-1)(k-2)}{2}\left(\frac{1}{\Gamma}\right)''(1-M) = -\alpha_M k(k-1)(k-2)(h_1(M-1)-\gamma)
    \end{displaymath}
    by Lemma \ref{gammaderivatives}.
    \item Case II: $m=k-2$.  In this case $k_0+\ldots +k_M=2$ which leaves three possibilities.
    \begin{enumerate}
        \item $k_0=1$, $k_i=1$ for a unique $i\ge 1$, and $k_j=0$ otherwise.  Thus $M=i$, and by \eqref{explicitsk}, Lemma \ref{gammaderivatives}, and Lemma \ref{hprop}, the resulting expression is
        \begin{displaymath}
         k(k-1)(k-2)(\beta-\gamma)\alpha_M.
         \end{displaymath}
         \item $k_0=0$, $k_i=2$ for some $i\ge 1$, and $k_j=0$ otherwise.  Thus $M=2i$ and we obtain the expression
         \begin{displaymath}
          \frac{1}{2}\frac{\alpha_i^2}{(i-1)!}k(k-1)(k-2)(-1)^{M-1}\frac{(M-1)!}{(i-1)!}.
         \end{displaymath}
         Note that $(-1)^{M-1}=-1$ because $M=2i$ is even.
         \item $k_0=0$, $k_i=k_j=1$ for some $i\ne j$, $i,j\ge 1$, and $k_{\ell}=0$ otherwise.  In this case $M=i+j$, and we obtain the expression
         \begin{displaymath}
          \frac{1}{2}\frac{(-1)^{i-1}\alpha_i}{(i-1)!}\frac{(-1)^{j-1}\alpha_j}{(j-1)!}k(k-1)(k-2)(-1)^{i+j-1}(M-1)!
         \end{displaymath}
         where the factor $1/2$ arises due to the cases $i<j$ and $i>j$ representing the same term.  Writing $j=M-i$ in the final two cases and summing over $i$, the contribution from the final two cases is
         \begin{displaymath}
          -\frac{1}{2}k(k-1)(k-2)\sum_{1\le i\le M-1}i\binom{M-1}{i}\alpha_i\alpha_{M-i}.
         \end{displaymath}
    \end{enumerate}
\end{enumerate}
This completes the proof.
\end{proof}

\section{Numbers with three prime factors}\label{R3sec}

In this section we establish Theorem \ref{R3thm1}.  By \cite[Proposition 3.2]{HME} we have
\begin{equation}\label{R3}
\mathcal{R}_3(x) = \frac{1}{3!}\left(\sum_{pqr\le x}\frac{1}{pqr} + 3\sum_{pq^2\le x}\frac{1}{pq^2}+2\sum_{p^3\le x}\frac{1}{p^3}\right).
\end{equation}
The leftmost sum in \eqref{R3} above equals $\mathcal{S}_3(x)$ and we use the fine-scale asymptotic given in Proposition \ref{s2s3} above.  Next, by Lemma \ref{primesquaretail}, the rightmost sum in \eqref{R3} is
\begin{displaymath}
\sum_{p\le x^{1/3}}\frac{1}{p^3} = P(3) - \sum_{p > x^{1/3}}\frac{1}{p^3} = P(3) + O(x^{-2/3})
\end{displaymath}
so that
\begin{align}\label{R33}
\mathcal{R}_3(x) = \frac{1}{6}\mathcal{S}_3(x) + \frac{1}{2}\sum_{pq^2\le x}\frac{1}{pq^2}+\frac{1}{3}P(3) + O(x^{-2/3}).
\end{align}

For the sum over $pq^2$ in \eqref{R33}, we follow the approach in \cite[Proposition 3.3]{HME}. Recall that $E(t):=\sum_{p\le t}p^{-1}-\log_2 t-\beta$, so by Mertens' second theorem
\begin{align}\label{eq:pq2x}
\sum_{pq^2\le x}\frac{1}{pq^2} = \sum_{q\le \sqrt{\frac{x}{2}}}\frac{1}{q^2}\sum_{p\le \frac{x}{q^2}}\frac{1}{p} = \sum_{q\le \sqrt{\frac{x}{2}}}\left(\frac{\log_2\frac{x}{q^2}}{q^2}+\frac{\beta}{q^2}+\frac{E\left(\frac{x}{q^2}\right)}{q^2}\right).
\end{align}
By Lemma \ref{primesquaretail}, we have
\begin{align}\label{eq:rootE}
\sum_{q\le \sqrt{\frac{x}{2}}}\frac{\beta}{q^2} = P(2)\beta + O_A(\log^{-A}x)
\end{align}
for all $A>0$.  Moreover, we claim
\begin{align}\label{eq:rootx2E}
\sum_{q\le \sqrt{\frac{x}{2}}}\frac{E(x/q^2)}{q^2} \ll_A\log^{-A}x
\end{align}
for all $A>0$. Indeed, by the prime number theorem
\begin{equation}\label{mertens2error}
E(x):= \sum_{p\le x}\frac{1}{p} -\log_2 x-\beta\ll_A \log^{-A}x.
\end{equation}
So by splitting the sum in \eqref{eq:rootx2E} at $x^{1/3}$,
\begin{displaymath}
\sum_{x^{1/3}<q\le \sqrt{x/2}}\frac{E(x/q^2)}{q^2} \ll \sum_{x^{1/3}<q}\frac{1}{q^2}\ll \frac{1}{x^{1/3}\log(x^{1/3})}\ll_A \log^{-A}x
\end{displaymath}
by \eqref{mertens2error} and Lemma \ref{primesquaretail}.  Also by \eqref{mertens2error} we have
\begin{displaymath}
\sum_{q\le x^{1/3}}\frac{E(x/q^2)}{q^2} \ll_A \sum_{q\le x^{1/3}}\frac{1}{q^2\log^A(x/q^2)}\ll_A \frac{1}{\log^A x}\sum_q \frac{1}{q^2}\ll_A\log^{-A} x
\end{displaymath}
for all $A>0$. These combine to give \eqref{eq:rootx2E}. Hence plugging \eqref{eq:rootE}, \eqref{eq:rootx2E} back into \eqref{eq:pq2x} gives
\begin{align}
\sum_{pq^2\le x}\frac{1}{pq^2} = \sum_{q\le \sqrt{\frac{x}{2}}}\frac{\log_2(x/q^2)}{q^2} \ +\  P(2)\beta + O_A(\log^{-A} x).
\end{align}

For the main term we proceed as in \cite[Proposition 3.3]{HME}. By partial summation, we have
\begin{align*}
\sum_{q\le \sqrt{\frac{x}{2}}}\frac{\log_2(x/q^2)}{q^2} &= s_2\left(\sqrt{\frac{x}{2}}\right)\log_2 2 + \int_2^{\sqrt{\frac{x}{2}}} \frac{2s_2(t) dt}{t\log(x/t^2)}
\end{align*}
where $s_a(t):=\sum_{p\le t}p^{-a}$.  Writing $s_2(t) = P(2) - \epsilon_2(t)$, the right-hand side is
\begin{align*}
P(2)\log_2 2 -\epsilon_2\left(\sqrt{\frac{x}{2}}\right)\log_2 2 + P(2)\int_2^{\sqrt{\frac{x}{2}}} \frac{2~dt}{t\log(x/t^2)} - \int_2^{\sqrt{\frac{x}{2}}} \frac{2\epsilon_2(t)~dt}{t\log(x/t^2)}.
\end{align*}
Evaluating the first integral above directly and simplifying, we obtain 
\begin{equation}\label{integralformula}
\begin{aligned}
\sum_{q\le \sqrt{\frac{x}{2}}}\frac{\log_2(x/q^2)}{q^2} &= P(2)\log_2\frac{x}{4}+O_N\left(\log^{-N}x\right)-\int_2^{\sqrt{\frac{x}{2}}} \frac{2\epsilon_2(t)~dt}{t\log(x/t^2)}
\end{aligned}
\end{equation}
for any $N>0$.  Here we used Lemma \ref{primesquaretail} which gives the estimate
\begin{displaymath}
 \epsilon_2\left(\sqrt{\frac{x}{2}}\right)\ll 1/(\sqrt{x}\log x)\ll_N \log^{-N}x.
\end{displaymath}
Now, by \eqref{logexpand2},
\begin{displaymath}
P(2)\log_2\frac{x}{4} = P(2)\log_2 x - P(2)\sum_{j\ge 1}\frac{(\log 4)^j}{j\log^j x}.
\end{displaymath}
It remains to bound the integral
\begin{displaymath}
I:=\int_2^{\sqrt{\frac{x}{2}}} \frac{2\epsilon_2(t)~dt}{t\log(x/t^2)}
\end{displaymath}
appearing in \eqref{integralformula}.  We use the expansion
\begin{equation}\label{logexpand3}
\frac{1}{\log(x/t^2)} = \sum_{1\le m<N}\frac{2^{m-1}\log^{m-1}t}{\log^m x}+O_N\left(\frac{\log^{N} t}{\log^{N}x}\right)
\end{equation}
which follows from \eqref{logexpand}.  Recalling the notation of Lemma \ref{djk}, we have
\begin{equation}\label{dm2}
\int_{2}^{\sqrt{\frac{x}{2}}}\frac{\epsilon_2(t)\log^{m-1}t}{t}dt = d_{m,2} + O_N\left(\frac{1}{\log^{N} x}\right)
\end{equation}
uniformly for $1\le m<N$ by Lemma \ref{primesquaretail}.  Combining \eqref{logexpand3} and \eqref{dm2}, we find that for any $N > 0$ we have
\begin{equation}
I = \sum_{1\le m<N} \frac{2^m d_{m,2}}{\log^m x} + O_N\left(\frac{1}{\log^{N}x}\right).
\end{equation}
Referring to Section \ref{notes} and combining all terms, we find that the $1/\log x$ coefficient of $\mathcal{R}_3(x)$ is given by 
\begin{displaymath}
\alpha_1(\log_2 x+\beta)-P(2)\log 2-d_{1,2}.
\end{displaymath}
Finally, by Lemma \ref{djk}, we have $P(2)\log 2+d_{1,2} = \alpha_{1,2}$.  A similar argument applies to the general term with denominator $\log^j x$.

In fact, we have the following proposition, which can be proved in the same way by applying partial summation with the function $f_j(t)=\log^j t$.

\begin{prop}\label{pq2}
For each integer $N>0$, we have
\begin{displaymath}
\sum_{pq^2\le x}\frac{1}{pq^2} = P(2)(\log_2 x+\beta) - \sum_{1\le j<N}\frac{2^j\alpha_{j,2}}{j\log^j x}+O_N\left(\frac{1}{\log^{N}x}\right).
\end{displaymath}
\end{prop}

It remains to verify the estimate \eqref{rnestimate}.  This follows directly from \eqref{vnestimate} (dividing by 6) along with Proposition \ref{alphajkestimate} with $k=2$, which gives $2^{j-1}\alpha_{j,2}/j = (1+o(1))2^{j-1}(j-1)!/j = (1+o(1))\alpha_j = o(\alpha_j\log j)$.  This completes the proof of Theorem \ref{R3thm1}.

\section{Numbers with four prime factors}\label{R4sec}

In this section, we prove a fine-scale asymptotic for $\mathcal{R}_4(x)$.

\begin{thm}\label{R4thm}
For any integer $N>0$, we have

\begin{displaymath}
\begin{aligned}
{\mathcal R}_4(x) &= \frac{1}{24}(\log_2 x+\beta)^4+\frac{P(2)-\zeta(2)}{4}(\log_2 x+\beta)^2
+\frac{P(3)+\zeta(3)}{3}(\log_2 x+\beta)\\
&\quad\qquad +\frac{P(4)}{4}+\frac{\zeta(4)}{16} + \frac{P(2)^2}{8}-\frac{P(2)\zeta(2)}{4}\\
   &\quad +\sum_{1\le j<N}\frac{\alpha_j((\log_2 x+\beta)^2+P(2)-\zeta(2))/2 - r_j(\log_2 x+\beta)-t_j}{\log^j x} \ + \ O_N\left(\frac{(\log_2 x)^2}{\log^{N} x}\right)
\end{aligned}
\end{displaymath}
for constants $t_j$.  Here $r_j$ is defined as in Theorem \ref{R3thm1}. Additionally, we have $t_1 = \alpha_{1,3}$ and $t_2 = 3\alpha_{2,3}/2-(\alpha_{2,2}+\alpha_1\alpha_{1,2}+\alpha_1^2/2)$.
\end{thm}

\begin{rmk}
In principle, the proof of Theorem \ref{R4thm} also gives a method for determining formulas for $t_j$ for any $j$.
\end{rmk}

\begin{rmk}
Note that $((\log_2 x+\beta)^2+P(2)-\zeta(2))/2 = R_{2,0}(\log_2 x)$.  In the notation of \cite{HME} we have $R_{k,0}(\log_2 x)=R_k(\log_2 x+\beta)$.
\end{rmk}

For the proof of Theorem \ref{R4thm}, as with \cite[Proposition 3.2]{HME}, we must handle five sums, i.e.
\begin{displaymath}
\mathcal{R}_4(x) = \frac{1}{4!}\left(\sum_{pqrs\le x}\frac{1}{pqrs} + 6\sum_{p^2qr\le x}\frac{1}{p^2qr}+3\sum_{p^2 q^2\le x}\frac{1}{p^2q^2}+8\sum_{p^3q\le x}\frac{1}{p^3 q} + 6\sum_{p^4\le x}\frac{1}{p^4}\right).
\end{displaymath}

The sum over $pqrs\le x$ is $\mathcal{S}_4(x)$, so we may use \eqref{s41} and \eqref{s42} in Section \ref{notes}.  The sum over $p^3q\le x$ is handled by the following generalization of Proposition \ref{pq2}, whose proof is analogous.  

\begin{prop}\label{pqa}
For each integer $a\ge 2$ and each integer $N>0$, we have
\begin{displaymath}
\sum_{pq^a\le x}\frac{1}{pq^a} = P(a)(\log_2 x+\beta) - \sum_{1\le j<N}\frac{a^j\alpha_{j,a}}{j\log^j x}+O_{a,N}\left(\frac{1}{\log^{N}x}\right).
\end{displaymath}
\end{prop}

The sum over $p^4\le x$ is $P(4)$ up to negligible error, and the sum over $p^2q^2\le x$ can be estimated as $P(2)^2$ up to negligible error by splitting the sum at $x^{1/4}$.  It remains then to deal with the sum

\begin{equation}\label{p2qr1}
\sum_{p^2qr\le x}\frac{1}{p^2qr} = \sum_{p\le \sqrt{\frac{x}{4}}}\frac{1}{p^2}
\,\mathcal{S}_2\left(\frac{x}{p^2}\right).
\end{equation}

We may apply the estimate for $\mathcal{S}_2(x)$ and write the sum in \eqref{p2qr1} as

\begin{equation}\label{p2qr2}
\sum_{p\le \sqrt{\frac{x}{4}}}\left(\frac{(\log_2 \frac{x}{p^2}+\beta)^2}{p^2}-\frac{\zeta(2)}{p^2} + \sum_{1\le j<N}\frac{2\alpha_j}{p^2\log^j\frac{x}{p^2}}+O_N\left(\frac{1}{p^2\log^{N}\frac{x}{p^2}}\right)\right).
\end{equation}

We bound the error term as $\ll_N \log^{-{N}}x$ by splitting the sum at $x^{1/4}$ and applying Lemma \ref{primesquaretail}. Next, note that by partial summation we have

\begin{displaymath}
\sum_{p\le \sqrt{\frac{x}{4}}}\frac{1}{p^2\log^j\frac{x}{p^2}} = \frac{P(2)}{\log^j\frac{x}{4}} + O\left(\frac{1}{\sqrt{x}\log x}\right) + j\int_2^{\sqrt{\frac{x}{4}}} \frac{2\epsilon_2(t)}{t\log^{j+1}\frac{x}{t^2}}dt.
\end{displaymath}

At this point, we can expand the logarithms recalling \eqref{logexpand3}, namely
\begin{displaymath}
\frac{1}{\log\frac{x}{t^2}} = \sum_{1\le m<N}\frac{2^{m-1}\log^{m-1}t}{\log^m x}+O_N\left(\frac{\log^{N} t}{\log^{N}x}\right),
\end{displaymath}
then raise the resulting expressions to the power $j$ using the multinomial theorem.  The remaining sums are similar to the ones above.  A slight modification of the argument handling \eqref{eq:pq2x} allows us to replace $\sqrt{x/2}$ with $\sqrt{x/4}$.  Specifically, by partial summation we have

\begin{displaymath}
\begin{aligned}
\sum_{p\le \sqrt{\frac{x}{4}}}\frac{(\log_2 \frac{x}{p^2})^2}{p^2} &= s_2\left(\sqrt{\frac{x}{4}}\right)(\log_2 4)^2 + \int_2^{\sqrt{\frac{x}{4}}}\frac{4s_2(t)\log_2\frac{x}{t^2}}{t\log\frac{x}{t^2}}dt\\
&= P(2)\left(\log_2\frac{x}{4}\right)^2 + O\left(\frac{1}{\sqrt{x}\log x}\right)-\int_2^{\sqrt{\frac{x}{4}}}\frac{4\epsilon_2(t)\log_2\frac{x}{t^2}}{t\log\frac{x}{t^2}}dt.
\end{aligned}
\end{displaymath}

We can now use the expansion
\begin{equation}\label{logexpand4}
\log_2\frac{x}{t^2} = \log_2 x - \sum_{j\ge 1}\frac{(2\log t)^j}{j\log^j x},
\end{equation}
which is a special case of \eqref{logexpand2} above,
and \eqref{logexpand3}.  Combining all estimates and taking Lemma \ref{djk} into account, we complete the proof of Theorem \ref{R4thm}.

Note that an alternative approach is to expand the logarithms to obtain an estimate for the sum in \eqref{p2qr2} up to error $\ll_N (\log_2 x)/\log^{N} x$ as
\begin{displaymath}
\sum_p p^{-2}\bigg(\big(\log_2 x - \sum_{1\le j<N} \frac{2^j\log^j p}{j\log^j x}+\beta\big)^2-\zeta(2)+\sum_{1\le j < N} 2\alpha_j \big(\sum_{1\le m < N}\frac{2^{m-1}\log^{m-1} p}{\log^m x}\big)^j\bigg).
\end{displaymath}

Rearranging terms, multiplying by $6/4!$, and adding the contribution from Proposition \ref{pqa} with $a=3$ (multiplying by $8/4!$), we obtain the following formula for the fine-scale terms in $\mathcal{R}_4(x)$ (aside from those in $\mathcal{S}_4(x)$ which are derived in Section \ref{notes}):
\begin{equation}\label{r4precisionterms}
 \begin{aligned}
 &-\frac{1}{2}(\log_2 x+\beta)\sum_{1\le j<N} \frac{2^{j}\alpha_{j,2}}{j\log^j x}-\frac{1}{3}\sum_{1\le j<N} \frac{3^j \alpha_{j,3}}{j\log^j x}+\frac{1}{4}\sum_p p^{-2}\left(\sum_{1\le j<N} \frac{2^{j}\log^j p}{j\log^j x}\right)^2\\
 &\qquad + \frac{1}{4}\sum_p  p^{-2}\sum_{1\le j<N} 2\alpha_j\left(\sum_{1\le m<N}\frac{2^{m-1}\log^{m-1} p}{\log^m x}\right)^j + O_N\left(\frac{\log_2 x}{\log^{N} x}\right).
 \end{aligned}
 \end{equation}

Combining Proposition \ref{k-3} with formula \eqref{r4precisionterms} gives Theorem \ref{R4thm}.

\section{Computations}\label{compute}

We produced the following numerical approximations of the ratio $\alpha_j/(2^jj!/2j^2)$ by implementing the formula \cite[(33)]{cricsan2021counting} of Cri{\c{s}}an and Erban using the computer program Pari/GP.  The table below illustrates the rate of convergence for the asymptotic estimate for the $\alpha_j$ given by Theorem \ref{mainthm}.

\begin{center}
    \begin{tabular}{c|c|c|c}
   $j$ & $\alpha_j /
  \left( \frac{j!\,2^{j}}{2j^2}\right)$
     & $j$ & $\alpha_j /
  \left( \frac{j!\,2^{j}}{2j^2}\right)$\\
     \hline
1 &  $1.332582$ & 14  & $1.012312$\\
2  & $1.277553$ & 15  & $1.009180$\\
3  & $1.281728$ & 16  & $1.006852$\\
4 &  $1.236091$ & 17  & $1.005118$\\
5  & $1.181314$ & 18  & $1.003825$\\
6  & $1.135188$ & 19  & $1.002861$\\
7  & $1.099931$ & 20  & $1.002140$\\
8  & $1.073764$ & 21  & $1.001602$\\
9  & $1.054504$ & 22  & $1.001199$\\
10  & $1.040343$ & 23  & $1.000898$\\
11  & $1.029915$ & 24  & $1.000673$\\
12  & $1.022220$ & 25 & $1.000504$\\
13  & $1.016529$ & 26 & $1.000377$
    \end{tabular}
\end{center}

\appendix

\section{Polylogarithm estimates} \label{polylogsec}

Let $M$ be fixed.  We want to show that for $x \in \lbrack 2, M\rbrack$,
\begin{displaymath}
\text{Li}_{1-j} \left( \frac{1}{x} \right) \ll_M \frac{(j-1)!}{(\log x)^j}.
\end{displaymath}

(We used this estimate with $M=e^{16}$ in the proof of Theorem \ref{mainthm}.)  We demonstrate that in fact, for fixed $x>1$ we have $\text{Li}_{-k}(1/x)\sim I(k,x)$, where
\begin{displaymath}
I(k,x) := \int_1^\infty \frac{t^k}{x^t}~dt
\end{displaymath}
and therefore $\text{Li}_{-k}(1/x)\sim k!/\log^{k+1}x$ as $k\to \infty$.  

We begin by mentioning a recurrence relation for this integral, as well as an explicit formula.  We then compare the sum $\text{Li}_{-k}\left(1/x\right)$ with the integral $I(k,x)$, verifying that they are asymptotic as $k\to \infty$.  The rate of convergence can be made explicit by using an explicit form of Stirling's formula.

\begin{prop}\label{polylogintegralformula}
Let $x>1$.  We have the recurrence relation
\begin{displaymath}
I(k,x) =\frac{1}{x\log x} + \frac{k}{\log x}I(k-1,x)~~(k\ge 1),~~I(0,x)=\frac{1}{x\log x}.
\end{displaymath}
It follows that for all $k\ge 0$, we have
\begin{displaymath}
I(k,x) = \frac{1}{x}\sum_{j=1}^{k+1}\frac{k!}{(k+1-j)!}\frac{1}{\log^j x}.
\end{displaymath}
Moreover,
\begin{equation} \label{I(k,x)bounds}
\frac{k!}{\log^{k+1}x}\left(1-\frac{\log^{k+1} x}{(k+1)!}\right) < I(k,x) < \frac{k!}{\log^{k+1}x}\left(1-\frac{\log^{k+1} x}{x(k+1)!}\right).
\end{equation}
Therefore, for fixed $x>1$ we have $I(k,x)\sim k!/\log^{k+1} x~~(k\to\infty)$.
\end{prop}

\begin{proof}
Under the change of variables $u=t\log x$, the recurrence relation is that of the upper incomplete gamma function.  The explicit formula then follows by induction on $k$.  Finally, note that
\begin{displaymath}
I(k,x)\Big/\frac{k!}{\log^{k+1}x} = \frac{1}{x}\sum_{j=1}^{k+1}\frac{\log^{k+1-j}x}{(k+1-j)!} = \frac{1}{x}\sum_{j=0}^k\frac{\log^j x}{j!} = \frac{e^{\log x}-Q_k(\log x)}{x},
\end{displaymath}
where
\begin{displaymath}
Q_k(v) := \sum_{j=k+1}^\infty \frac{v^j}{j!}
\end{displaymath}
is the remainder in the Taylor series for the exponential function.  We have
\begin{displaymath}
\frac{\log^{k+1}x}{(k+1)!}<Q_k(\log x)<\frac{x\log^{k+1}x}{(k+1)!},
\end{displaymath}
where the lower bound is obtained by retaining only one term, and the upper bound is given by Taylor's formula for the remainder term.  (Specifically, for some $z\in (0,\log x)$, we have $Q_k(\log x)=e^z \log^{k+1} x/(k+1)!$.)

Therefore,
\begin{displaymath}
1-\frac{\log^{k+1}x}{(k+1)!}<I(k,x)\Big/\frac{k!}{\log^{k+1}x} < 1-\frac{\log^{k+1}x}{x(k+1)!}.
\end{displaymath}
This completes the proof of Proposition \ref{polylogintegralformula}.
\end{proof}

We now relate the sum and the integral for fixed $x>1$, establishing their asymptotic equality as $k\to \infty$.  We will also use the following explicit form of Stirling's formula. (See for instance \cite[Remark 1.15]{lucadekoninck}.)

\begin{lem}\label{stirling}
For all $k\ge 1$, we have
\begin{displaymath}
1\le \frac{k!}{\left(\frac{k}{e}\right)^k\sqrt{2\pi k}}\le e^{1/(12k)}.
\end{displaymath}
\end{lem}

\begin{prop}\label{Sumvsintegral}
For all $x>1$ and $k\ge 1$, we have
\begin{displaymath}
\left|\frac{{\rm Li}_{-k}\left(1/x\right)}{I(k,x)} - 1\right|<\frac{x}{\sqrt{2\pi k}}.
\end{displaymath}
In particular, for fixed $x>1$ we have ${\rm Li}_{-k}\left(1/x\right)\sim I(k,x)~~(k\to \infty)$.
\end{prop}

\begin{proof}
By Euler's summation formula, we have

\begin{displaymath}
\begin{aligned}
\text{Li}_{-k}\left(\frac{1}{x}\right) &= \frac{1}{x} + \lim_{z\to \infty}\sum_{1<n\le z}\frac{n^k}{x^n}\\
& = \frac{1}{x} + \lim_{z\to \infty}\left(\int_1^z \frac{t^k}{x^t}~dt + \int_1^z \frac{t^{k-1}(k - t\log x)}{x^t}(t-\lfloor t\rfloor)~dt + \frac{z^k}{x^z}(\lfloor z\rfloor - z)\right)\\
& = \frac{1}{x} + I(k,x) + \int_1^\infty \frac{t^{k-1}(k - t\log x)}{x^t}(t-\lfloor t\rfloor)~dt.
\end{aligned}
\end{displaymath}
We now bound the integral on the right and show that it is $o(I(k,x))$ as $k\to \infty$ by splitting the interval at $k/\log x$.  We have $0\le t-\lfloor t\rfloor < 1 $, so the integrand is nonnegative for $1\le t\le k/\log x$ and nonpositive for $t\ge k/\log x$.  Letting $f(t) = t^k/x^t$, we thus bound

\begin{displaymath}
0\le \int_1^{k/\log x} \frac{t^{k-1}(k - t\log x)}{x^t}(t-\lfloor t\rfloor)~dt \le \int_1^{k/\log x} f'(t)~dt = f\left(\frac{k}{\log x}\right) -\frac{1}{x}
\end{displaymath}

and 

\begin{displaymath}
\begin{aligned}
0\ge \int_{k/\log x}^\infty \frac{t^{k-1}(k - t\log x)}{x^t}(t-\lfloor t\rfloor)~dt &= -\int_{k/\log x}^\infty -\frac{t^{k-1}(k - t\log x)}{x^t}(t-\lfloor t\rfloor)~dt\\
&\ge -\int_{k/\log x}^\infty -f'(t) dt = -f\left(\frac{k}{\log x}\right).
\end{aligned}
\end{displaymath}

Therefore,
\begin{displaymath}
\left|\int_1^\infty \frac{t^{k-1}(k - t\log x)}{x^t}(t-\lfloor t\rfloor)~dt\right| \le f\left(\frac{k}{\log x}\right) = \frac{\left(\frac{k}{\log x}\right)^k}{x^{k/\log x}} = \frac{\left(\frac{k}{e}\right)^k}{\log^k x},
\end{displaymath}
noting that $x^{k/\log x} = e^k$.  On the other hand,
\begin{displaymath}
 I(k,x) > \frac{k!}{x\log^k x} > \frac{\left(\frac{k}{e}\right)^k\sqrt{2\pi k}}{x\log^k x} 
\end{displaymath}
by Proposition \ref{polylogintegralformula} and the lower bound of Lemma \ref{stirling}.  Thus $$\frac{\left(\frac{k}{e}\right)^k}{\log^k x}\Big/ I(k,x) < \frac{x}{\sqrt{2\pi k}}.$$
Noting that for the upper bound we may cancel out the two terms $1/x$ and $-1/x$, we complete the proof of Proposition \ref{Sumvsintegral}.
\end{proof}

\section*{Acknowledgments}
We thank Ofir Gorodetsky for sketching an alternate proof of the asymptotic relation in Theorem \ref{mainthm} via the von Mangoldt explicit formula. We are also grateful to Michael Morris and Samanthak Thiagarajan for helpful discussions involving polylogarithms. The third author was supported in part by a G-Research grant, and a Clarendon Scholarship at the University of Oxford.

\bibliographystyle{plain}

\end{document}